\documentclass[12pt,dvipdfmx]{amsart}
\usepackage{hyperref}
\hypersetup{%
 bookmarks=true,%
 bookmarksnumbered=true,%
 colorlinks=true,%
 allcolors=blue
}

\usepackage{amssymb} 
\usepackage{amsthm}
\usepackage{amscd}
\usepackage{multicol}
\usepackage[dvips]{color}
\usepackage[all]{xy}

\newtheorem{thm}{Theorem}[section]
\newtheorem{lem}[thm]{Lemma}
\newtheorem{lem-dfn}[thm]{Lemma-Definition}
\newtheorem{prop}[thm]{Proposition}
\newtheorem{cor}[thm]{Corollary}

\newtheorem{ass}[thm]{Assumption}

\theoremstyle{definition}
\newtheorem{defn}[thm]{Definition}

\newtheorem{ex}[thm]{Example}

\newtheorem{prob}[thm]{Problem}

\newtheorem*{acknowledgement}{Acknowledgement}
\theoremstyle{remark}
\newtheorem{clm}[thm]{Claim}
\newtheorem{rem}[thm]{Remark}
\numberwithin{equation}{section}

\newcommand{\thmref}[1]{Theorem~\ref{#1}}
\newcommand{\lemref}[1]{Lemma~\ref{#1}}
\newcommand{\corref}[1]{Corollary~\ref{#1}}
\newcommand{\proref}[1]{Proposition~\ref{#1}}
\newcommand{\remref}[1]{Remark~\ref{#1}}

\newcommand{\clmref}[1]{Claim~\ref{#1}}
\newcommand{\defref}[1]{Definition~\ref{#1}}

\newcommand{\exref}[1]{Example~\ref{#1}}

\newcommand{\figref}[1]{Figure~\ref{#1}}
\newcommand{\sref}[1]{Section~\ref{#1}}

\DeclareMathOperator{\ann}{Ann}

\DeclareMathOperator{\Spec}{Spec}
\DeclareMathOperator{\spec}{Spec}

\DeclareMathOperator{\proj}{Proj}

\DeclareMathOperator{\supp}{Supp}

\DeclareMathOperator{\Hom}{Hom}

\DeclareMathOperator{\Coker}{Coker}

\DeclareMathOperator{\pic}{Pic}

\DeclareMathOperator{\di}{div}

\DeclareMathOperator{\br}{\bar r}

\DeclareMathOperator{\gon}{gon}

\newcommand{\m}{\mathfrak m}

\newcommand{\PP}{\mathbb P}

\newcommand{\Z}{\mathbb Z}
\newcommand{\Q}{\mathbb Q}

\newcommand{\C}{\mathbb C}

\newcommand{\cB}{\mathcal B}

\newcommand{\cE}{\mathcal E}
\newcommand{\cF}{\mathcal F}

\newcommand{\cL}{\mathcal L}

\newcommand{\cO}{\mathcal O}

\renewcommand{\:}{\colon}

\newcommand{\ol}[1]{\overline {#1}}

\newcommand{\fl}[1]{\left\lfloor #1 \right\rfloor}

\newcommand{\mc}[1]{#1_{\mathrm{min}}}
\newcommand{\red}[1]{#1_{\mathrm{red}}}

\newcommand{\defset}[2]{{\left\{#1\,\left| \,#2 \right. \right\}}}

\begin{document}
\title[R{\"o}hr's vanishing theorem and the normal reduction number]{A variant of R{\"o}hr's vanishing theorem with an application to the normal reduction number for normal surface singularities}

\author{Tomohiro Okuma}
\address[Tomohiro Okuma]{Department of Mathematical Sciences, 
Yamagata University,  Yamagata, 990-8560, Japan.}
\email{okuma@sci.kj.yamagata-u.ac.jp}
\author{Kei-ichi Watanabe}
\address[Kei-ichi Watanabe]{Department of Mathematics, College of Humanities and Sciences, 
Nihon University, Setagaya-ku, Tokyo, 156-8550, Japan and 
Organization for the Strategic Coordination of Research and Intellectual Properties, Meiji University
}
\email{watnbkei@gmail.com}
\author{Ken-ichi Yoshida}
\address[Ken-ichi Yoshida]{Department of Mathematics, 
College of Humanities and Sciences, 
Nihon University, Setagaya-ku, Tokyo, 156-8550, Japan}
\email{yoshida.kennichi@nihon-u.ac.jp}
\thanks{TO was partially supported by JSPS Grant-in-Aid 
for Scientific Research (C) Grant Number 21K03215.
KW  was partially supported by JSPS Grant-in-Aid 
for Scientific Research (C) Grant Number 23K03040.
KY was partially supported by JSPS Grant-in-Aid 
for Scientific Research (C) Grant Number 24K06678.}

\begin{abstract} 
Let $A$ be an excellent two-dimensional normal local ring containing an algebraically closed field and let $X\to \spec (A)$ be a resolution of singularity.
We prove a theorem giving a condition under which the dimension of the cohomology group of invertible sheaves on $X$ coincides with a natural lower bound.
Applying this theorem, we establish upper bounds for the normal reduction number $\br(A)$ of $A$. 
For example, we prove the inequality $\br(A) \le p_a(A)+1$, where $p_a(A)$ denotes the arithmetic genus, a fundamental combinatorial (topological) invariant.
We introduce the notion of almost cone singularities and give a sharper
 inequality $\br(A) \le p_f(A)+1$ for such singularities, 
where $p_f(A)$ denotes the fundamental genus.
We also show that $\br(A)$ is not a combinatorial invariant in general.
\end{abstract}

\maketitle

\section{Introduction}
Let $(A, \m, k)$ be an excellent two-dimensional normal local domain containing an algebraically closed field isomorphic to the residue field $k$. 
In this paper, we simply call such a local ring a {\em normal surface singularity}.
Let $X\to \spec (A)$ be a resolution of the singularity with exceptional set $E=\bigcup_{i=1}^n E_i$, where $E_i$ are irreducible components. 
We call a divisor $C$ on $X$ a cycle if $C$ is supported in $E$, and write $\chi(C)=\chi(\cO_C)$.
It is known that many important invariants and properties are obtained from the cohomology of invertible sheaves on $X$.
There are various types of vanishing theorems for invertible sheaves on $X$ and they play essential roles in the study of singularities. Laufer's vanishing theorem (\cite{la.rat}), which is a generalization of the Grauert-Riemenschneider vanishing theorem (\cite{g-r.vanish}), is one of the useful theorems and states that for an invertible sheaf $\cL$ on $X$, we have $H^1(\cL)=0$ if $\cL E_i\ge K_XE_i$ for every $E_i \subset E$, where $K_X$ denotes the canonical divisor on $X$.
Note that $K_XC=-2\chi(C)-C^2$ by the Riemann-Roch formula, and $-C^2>0$ for any cycle $C>0$.
Let $\cB$ denote the set of all cycles that appear in a computation sequence for the fundamental cycle $Z_f$ (see. \sref{s:CCC}).
R{\"o}hr \cite{Ro} proved the following.

\begin{thm}[R{\"o}hr]
\label{t:rohrV}
$H^1(\cL)=0$ if $\cL C> -2\chi(C)$ for every cycle $C\in \cB$.
\end{thm}
We have seen that R{\"o}hr's vanishing theorem is very useful in the study of elliptic singularities and cone singularities (e.g.,  \cite{Ok1}, \cite{Ok2}, \cite{OWY5}, \cite{ORWY}).
In the context of normal surface singularities, it is often important to determine the dimension of cohomology groups for invertible sheaves on resolution spaces, regardless of whether they vanish.
Motivated by this, we extend R{\"o}hr's vanishing theorem under specific conditions.
For a divisor $L$ on $X$, let $L^{\bot}=\bigcup_{LE_i=0} E_i$, and let $p_g(L^{\bot})$ denote the sum of the geometric genus 
of the singularities obtained by contracting $L^{\bot}$ (see \sref{s:Pre}).
If $\cO_X(L)$ has no fixed components, then $h^1(\cO_X(L))\ge p_g(L^{\bot})$ (see \lemref{l:nD}).

 Let $\cB(L)^e$ denote the set of all cycles $C>0$ with the property that $C$ has a decomposition $C=C_1+C_2$ such that $C_1\in \cB$, $C_2\ge 0$,  $LC_1>0$, $LC_2=0$, and $C_1$ is anti-nef on $C_2$ (see \sref{s:R}).
Then, we have the following.

\begin{thm}[see \thmref{t:Main}]
\label{t:Main0}
Assume that $\cO_X(L)$ has no fixed components and that 
 $LC> -2\chi(C)$ for every $C\in \cB(L)^e$.
Then $h^1(\cO_X(L))=p_g(L^{\bot})$. 
\end{thm}

Note that $p_g(L^{\bot})=0$ if and only if $L^{\bot}$ contracts to rational singularities. 
By Artin's results \cite{Ar-rat}, the condition $p_g(L^{\bot})=0$ is equivalent to $\chi(C)>0$ for every $C>0$ such that $\supp(C)\subset \supp(L^{\bot})$.
Since the cycle $C=C_1+C_2$ in \thmref{t:Main0} satisfies $\chi(C)=\chi(C_1)+\chi(C_2)-C_1C_2\ge \chi(C_1)+\chi(C_2)$,  \thmref{t:Main0} deduces R{\"o}hr's vanishing theorem for $\cO_X(L)$ without fixed components.

In \cite{OWY4}, the authors introduced the {\em normal reduction number} $\br$ and proved its basic properties.
For any integrally closed $\m$-primary ideal $I \subset A$ and 
its minimal reduction $Q$, we put  
\begin{align*}
\br(I)&=\min\defset{r \in \Z_{>0} }{ \overline{I^{n+1}}=Q \overline{I^n} 
\; \text{for all $n \ge r$}},
\\
\br(A)&=\max\defset{\br(I)}{ I \subset A: \text{  integrally closed $\m$-primary ideal}}.
\end{align*}
These invariants can also be expressed in terms of cohomology (see \proref{p:brq}).
We will apply \thmref{t:Main0} to obtain upper bounds for the normal reduction number of surface singularities (see \thmref{t:lambda}).

By \cite[\S 2.1]{OWY4}, we have $\br(A) \le p_g(A)+1$, where $p_g(A)=\dim_k H^1(\cO_X)$, the geometric genus of $A$.
We expect that the normal reduction number may play an important role as a fundamental invariant of normal surface singularities.  For example, we see that $A$ is a rational singularity if and only if $\br(A)=1$. 
However, computing the normal reduction number has been difficult, and no good combinatorial bounds have been known; here, an invariant of $A$ is said to be {\em combinatorial} (or, {\em topological}) if it is determined by the resolution graph of $A$.
Recall that the fundamental genus $p_f(A)=1+Z_f(Z_f+K_X)/2$ and the arithmetic genus $p_a(A)=\max\defset{1-\chi(C)}{\text{$C>0$ is a cycle on $X$}}$ are 
  very fundamental combinatorial invariants of $A$ and satisfy $p_f(A) \le p_a(A) \le p_g(A)$. 
In general, the differences 
$p_g(A)-p_a(A)$ and $p_a(A)-p_f(A)$  can take various positive integers; for example,
if $A=k[[x,y,z]]/(f)$, $f$ is a homogeneous polynomial of degree $d$, then $p_f(A)=(d-1)(d-2)/2$, $p_a(A)=1+du(d-u-2)/2$, where $u=\fl{d/2}-1$, and $p_g(A)=d(d-1)(d-2)/6$. 
Moreover, 
 for any $n\in \Z_{>0}$, there exists a singularity $A$ with $p_a(A)=1$ and $p_g(A)=n$ (cf. \exref{e:KYC}).
Our arguments in the proof of \thmref{t:Main0} provides an alternative proof of the following.

\begin{thm}[\cite{NNO} (see \thmref{t:brp})]
\label{t:brp0}
$\br(A) \le p_a(A)+1$.
\end{thm}
This theorem immediately implies that $\br(A) = 2$ if $A$ is an elliptic singularity, namely, if $p_a(A)=1$ (cf. \cite[Theorem 3.3]{Ok2}).
We provide a counter-example to the converse of this result (\exref{e:br346}).
Consequently, we obtain partial answers to \cite[Problem 3.1-3.2]{o.NRN}.

We introduce almost cone singularities (\sref{s:ACS}), which generalize cone singularities, and prove the following.

\begin{thm}[see \thmref{t:ACmain}]
\label{t:conepf}
Let $A$ be an almost cone singularity.
Then $\br(A) \le p_f(A)+1$.
Note that $p_f(A)$ coincides with the genus of the central curve.
\end{thm}

This paper is organized as follows.
In \sref{s:Pre}, we summarize the terminology and basic notions concerning cycles on a resolution of singularities.
This also includes a brief overview of the chain-connected component decomposition of cycles introduced by Konno \cite{Ko-CC}, which serves as a preparation for \thmref{t:Main0}.
In \sref{s:vanish}, we prove \thmref{t:Main0}.
To this end, we consider the contraction $f\: X\to Y$ of $L^{\bot}$ and reduce the problem to the vanishing of the cohomology of $f_*L$ on the normal surface $Y$. 
We introduce a pull-back $f^+W$ of an arbitrary cycle $W$ on $Y$, and reformulate the condition for the vanishing in terms of $f^+$ in \thmref{t:Main0}. 
In \sref{s:nr}, we apply the results in \sref{s:vanish} to establish upper bounds of the normal reduction number (\thmref{t:lambda}, \thmref{t:brp}).
In \sref{s:ACS}, we introduce the notion of {\em almost cone singularities} and prove \thmref{t:conepf}.
In \sref{s:example}, we prove that $\br(A)=2$ for a non-elliptic hypersurface singularity (\exref{e:br346}), and show that there exist cone singularities with the same resolution graph but different normal reduction number (\exref{e:non-topol}). Then it follows that $\br(A)$ is not a combinatorial invariant in general.
From the results mentioned above, it is natural to ask if $\br(A) \le p_f(A)+1$ holds in general. 
We also give an explicit example to show that this inequality does not hold in general (\exref{e:tomari}).

\begin{acknowledgement}
The authors are grateful to Professor Masataka Tomari for letting us know the singularity in \exref{e:tomari}.
\end{acknowledgement}

\section{Preliminaries}
\label{s:Pre}

Let $(A,\m,k)$ be a {\em normal surface singularity}, namely, an excellent two-dimensional normal local ring containing an algebraically closed field isomorphic to the residue field $k$.
Let $\pi\:X \to \spec (A)$ be a resolution of singularities with exceptional 
set $E:=\pi^{-1}(\m)$ and let $E=\bigcup_{i=1}^n E_i$ be the decomposition into the irreducible components.
A divisor on a resolution space whose support is contained in the exceptional set is called a {\em cycle}.
Let $\cE$ denote the set of all cycles on $X$ and $\cE^+$ the set of all positive cycles; 
\[
\cE=\sum_{i=1}^n \Z E_i, \quad \cE^+=\defset{C\in \cE}{C>0}.
\]
The support of a divisor $D$ is denoted by $\supp(D)$.
We say that a divisor $L$ on $X$ is {\em nef} on a cycle $C>0$ if $LE_i \ge 0$ for all $E_i \subset \supp(C)$; we simply say that $L$ is {\em nef} if it is nef on $E$.
The divisor $L$ is said to be {\em anti-nef} (on $C$) if $-L$ is nef (on $C$).
It is known that the set 
\[
\defset{Z\in \cE^+}{\text{$\supp(Z)=\supp(C)$ and $Z$ is anti-nef on $C$}}
\]
has a minimum, which is called the {\em fundamental cycle on $C$}; in case that $C=E$, we simply call it the {\em fundamental cycle} and denote it by $Z_f$.
A cycle $F > 0$ is called the {\em fixed part} of the divisor $L$ or the invertible sheaf $\cO_X(L)$ on $X$ if $F$ is the maximal cycle such that $H^0(\cO_X(L-F))=H^0(\cO_X(L))$.  
An irreducible component of the fixed part $F$ is called a {\em fixed component}.
For example, if $C>0$ is a cycle, then $-C$ has no fixed components if and only if there exists an element $h\in H^0(\cO_X(-C))$ such that $\di_X(h)=C+H$, where $H$ does not contain any component of $E$; 
if this is the case, we have that for a cycle $W>0$ with $CW=0$,
\begin{equation}
\label{eq:trivialW}
\cO_W(-C) \cong \cO_W(H)=\cO_W.
\end{equation}

 For a cycle $C\in \cE^+$, we regard it as a scheme with structure sheaf $\cO_C:=\cO_X/\cO_X(-C)$, and write 
\[
\chi(C)=\chi(\cO_C):=h^0(\cO_C)-h^1(\cO_C),
\]
 where $h^i(\cF)$ denotes $\dim_k H^i(\cF)$.
We define the arithmetic genus $p_a(C)$ of $C$ by  
\[
p_a(C)=1-\chi(C).
\]
By the Riemann-Roch formula, we have $\chi(C)=-C(C+K_X)/2$, where $K_X$ is the canonical divisor on $X$, and thus we obtain 
\begin{equation}
\label{eq:chi}
\chi(C+C')=\chi(C)+\chi(C')-CC'.
\end{equation}
We define the {\em fundamental genus} $p_f(A)$ of $A$ by $p_f(A) =p_a(Z_f)$,
the {\em arithmetic genus} $p_a(A)$ of $A$ by $p_a(A)=\max\defset{p_a(C)}{C\in \cE^+}$,
and the {\em geometric genus} $p_g(A)$ of $A$ by $p_g(A)=h^1(\cO_X)$.
These three genera are independent of the choice of resolution.
While $p_f(A)$ and $p_a(A)$ are combinatorial invariants, 
$p_g(A)$ cannot be determined by the resolution graph in general.
By the formal function theorem, $p_g(A)=\max\defset{h^1(\cO_C)}{C\in \cE^+}$.
 It is known that 
\begin{equation}\label{eq:ppp}
p_f(A) \le p_a(A) \le p_g(A).
\end{equation}

\begin{defn}
The singularity $A$ is called a {\em rational} singularity if $p_g(A)=0$, and an {\em elliptic} singularity if $p_f(A)=1$.
\end{defn}

It is known that $A$ is rational if and only if $p_f(A)=0$ (see \cite{Ar-rat}) and that $A$ is elliptic if and only if $p_a(A)=1$ (see \cite{wag.ell},  \cite{tomari.ell}).

The following example is a special case of \cite[Example 3.10]{Ko-YC}.

\begin{ex}\label{e:KYC}
Let $p, m\in\Z_{>0}$ and $B=k[x,y,z]_{(x,y,z)}$.
Assume that 
\[
A=B/(x^2 + y^{2p+1} + z^{2(2p+1)m}).
\]
Then we have the following:
\[
p_f(A)=p, \quad p_a(A)=mp(p-1)/2 + 1, \quad p_g(A)=p^2m.
\]
By the arguments in \cite{K-N}, we see that if $X$ is the minimal resolution,  $E$ is a chain of nonsingular curves $E_1, \dots, E_m$ with weighted dual graph as in \figref{fig:pfp}, and $\di_X(z)-H=E=Z_f$, where $H$ is the proper transform of the curve $\spec(A/(z)) \subset \spec (A)$.
\begin{figure}[htb]
\[
\xy
(0,0)*+{-1}*+!U(-2.0){E_1}*+!D(-2.0){[p]}*\cir<10pt>{}="A"; (15,0)*+{-2}*+!U(-2.0){E_2}*\cir<10pt>{}="B", (55,0)*+{-2}*+!U(-2.0){E_m} *\cir<10pt>{}="C", 
(30,0)*+{\cdot}; (35,0)*+{\cdot}; (40,0)*+{\cdot};   
\ar @{-} "A" ;"B"  \ar @{-}"B";(25,0)  \ar @{-} (45,0); "C" 
\endxy
\]
\caption{\label{fig:pfp} The resolution graph of $x^2 + y^{2p+1} + z^{2(2p+1)m}=0$}
\end{figure}
\end{ex}

\subsection{The CCC decomposition}\label{s:CCC}
We will provide a brief overview of the chain-connected component decomposition of cycles  introduced by Konno \cite{Ko-CC}.

It is known that the fundamental cycle $Z_f$ can be computed 
via a sequence $\{C_1, \dots, C_m\}$ of cycles such that
\[
 C_1=E_{j_1}, \quad C_i=C_{i-1}+E_{j_i} \ (1< i \le m), \quad C_m=Z_f,
\]
where $E_{j_1}$ is an arbitrary component of $E$ and $C_{i-1}E_{j_i}>0$ 
for $1< i \le m$.
Such a sequence $\{C_i\}_i$ is called a  
 {\em computation sequence} for $Z_f$. 
Let $\cB$ denote the set of all cycles appearing in a computation sequence for $Z_f$.
Then we have the following (see \cite[(2.6), (2.7)]{la.me}): 
\begin{equation}\label{eq: h1B}
h^0(\cO_C)=1, \quad p_a(C)\le p_f(A) \quad \text{for $C\in \cB$}.
\end{equation}

\begin{defn}[{\cite[1.3, 1.6]{Ko-CC}}]
Let $D \in \cE^+$.
Then $D$ is said to be {\em chain-connected} if there does not exist $D_1\in \cE^+$ such that $0< D_1 < D$ and $D_1$ is anti-nef on $D-D_1$.
The maximal chain-connected cycle $C \le D$ is called the {\em chain-connected component} of $D$. Note that the support of a chain-connected cycle is connected.
\end{defn}

\begin{lem}[{\cite[1.2]{Ko-CC}}]
Let $D\in \cE^+$. Then $D\in \cB$ if and only if $D$ is chain-connected.
\end{lem}

For a given cycle $D\in \cE^+$, there exists a chain-connected component $C_1$ of $D$.
If $D-C_1>0$, then we can take a chain-connected component $C_2$ of $D-C_1$.
Repeating this procedure we obtain a decomposition of $D$ into the sum of chain-connected cycles.

\begin{thm}
[Konno {\cite[1.7]{Ko-CC}}]
\label{t:K}
Let $D\in \cE^+$. 
Then there exist a sequence $D_1, \dots, D_r$ of distinct chain-connected cycles with $D_i \le D$ and a sequence $m_1, \dots, m_r$ of positive integers that satisfy
the following conditions$:$
\begin{enumerate}
\item $D=m_1D_1+\cdots + m_rD_r$.
\item For $i<j$, either $\supp(D_i) \cap \supp(D_j) = \emptyset$ or $D_i\ge D_j$.
\item For $i<j$, $D_i$ is anti-nef on $D_j$.
\item If $m_i\ge 2$, then $D_i$ is anti-nef on $D_i$.
\end{enumerate}
Moreover, the sequences above are unique up to a suitable permutation of the indices. 
\end{thm}

\begin{defn}
The ordered decomposition $D=m_1D_1+\cdots +m_rD_r$ as in \thmref{t:K} is called a {\em  chain-connected component decomposition} or a {\em CCC decomposition} of $D$.
\end{defn}

\begin{ass}
In the following,  when we say that $D=m_1D_1+\cdots +m_rD_r$ is a CCC decomposition, we assume that this expression satisfies the conditions in \thmref{t:K}.
\end{ass}

\subsection{ A modification of R{\"o}hr's lemma}\label{s:R}
Let us fix a divisor $L$ on $X$ which is not numerically trivial, 
that is, $LE_i\ne 0$ for some $E_i\le E$.
Let 
\[
L^{\bot} = \bigcup_ {E_i \subset E, \, LE_i=0} E_i, 
\] 
and define subsets $\cE^+(L)$ and $\cB(L)$ by
\[
\cE^+(L) = \defset{C\in \cE^+}{LC\ne 0 },   \ \ 
\cB(L)=\cE^+(L)\cap \cB.
\]
For a cycle $C\in \cE^+$, we consider the following condition:
\begin{itemize}
\item[$(\#)$]
There exist effective cycles $C_1$ and $C_2$ such that
$C=C_1+C_2$,  $C_1\in \cB(L)$, $\supp(C_2) \subset L^{\bot}$, 
and $C_1$ is anti-nef on $C_2$.
\end{itemize} 
We define 
 the set $\cB(L)^e$ by 
\[
\cB(L)^e = \defset{C \in \cE^+(L)}{ \text{ $C$ satisfies $(\#)$ }}.
\]
Then we have $\cB(L) \subset \cB(L)^e$ by the definition. 

\begin{rem}
Suppose we have a decomposition 
\[ C = C_1 + C_2 \quad \text{with } C_1 \in \cB(L),\ \supp(C_2) \subset L^{\bot}. \]
By applying \thmref{t:K} to the cycle $C$, 
or just taking a chain-connected component of $C$, 
we may choose $C_1$ and $C_2$ so that, in addition, 
$C_1$ is anti-nef on $C_2$.
\end{rem}

\begin{ex}
Let us consider the singularity in \exref{e:KYC} with $m=3$.
Assume that $X$ is the minimal resolution and let $L=-E=-(E_1+E_2+E_3)$. Since $E$ is the exceptional part of $\di_X(z)$, $L$ has no fixed components.
We have 
\[
L^{\bot}=E_1 \cup E_2, \quad \cB(L)=\defset{C\in \cB}{C\ge E_3}=\{E_3, E_2+E_3, E\}.
\]
An element of $\cB(L)^e$ is expressed as $E_3+ a_2E_2 +a_1 E_1$ $(a_1, a_2\in \Z_{\ge 0})$.
\end{ex}

The following proposition is a modified version of R{\"o}hr's lemma \cite[Lemma 1.3]{Ro}; by replacing $\cB(L)^e$ with $\cB$,  and $\cE^+(L)$ with $\cE^+$, it becomes the original situation.

\begin{prop}\label{p:Rohr}
Assume that $LD> -2\chi(D)$ for all $D\in \cB(L)^e$.
Then $LD> -2\chi(D)$ holds for all $D\in \cE^{+}(L)$.
\end{prop}
\begin{proof}
Let $\eta(D)=(-1/2)LD$ for $D\in \cE^+$. 
Then the inequality in the claim becomes $\eta(D) < \chi(D)$.
Assume that $D\in \cE^{+}(L)$ and let $D=m_1D_1+\cdots + m_rD_r$ be the CCC decomposition.
By the formula \eqref{eq:chi} for $\chi$, we may assume that $\supp(D)$ is connected.
Then, $D_1 \ge D_i$ for $i > 1$ since $\supp(D_1)=\supp(D)$,  and $\supp (D_1) \not\subset L^{\bot}$.
Moreover, changing the indices if necessary, 
we may assume that $\supp (D_i) \not\subset L^{\bot}$, namely, $D_i \in \cB(L)$ for $i \le p$,  and $\supp (D_i) \subset L^{\bot}$ for $i>p$.
Let  $D'= D_1+m_{p+1}D_{p+1} + \cdots + m_rD_r$.
Then this expression  is also 
a CCC decomposition, namely, the conditions (1)--(4) of \thmref{t:K} are satisfied.
Hence $D' \in \cB(L)^e$.
The conditions of \thmref{t:K} implies that $-D_iD_j\ge 0$ for $i\ne j$, and thus $-D'D_i\ge 0$ for $1\le i \le r$.
By the assumption, we have $\eta(D') < \chi(D')$ and $\eta(D_i) < \chi(D_i)$ for $i \le p$, since $\cB(L) \subset \cB(L)^e$.
Using the formula \eqref{eq:chi},  we obtain 
\begin{align*}
\chi(D) &=\chi(D' + (m_1-1)D_1 + m_{2}D_{2} + \cdots + m_pD_p) \\
&\ge \chi(D') + \chi((m_1-1)D_1) + \chi(m_{2}D_{2}) + \cdots + \chi(m_pD_p) \\
& \ge \chi(D') + (m_1-1) \chi(D_1) + m_2\chi(D_{2}) + \cdots + m_p\chi(D_p) \\
& > \eta (D') + (m_1-1) \eta(D_1) + m_2\eta(D_{2}) + \cdots + m_p\eta(D_p) \\
& = \eta(D' + (m_1-1)D_1 + m_{2}D_{2} + \cdots + m_pD_p)\\
& = \eta(D).
\qedhere
\end{align*}
\end{proof}

\subsection{The cohomology of invertible sheaves without fixed components}
We briefly summarize some results of \cite{Ok2} and \cite{OWY2}  on the cohomology of invertible sheaves on $X$.

\begin{defn}
For a cycle $D\in \cE^+$, let 
\[
p_g(D)=\max\defset{h^1(\cO_C)}{C\in \cE^+, \, \supp(C) \subset \supp(D)}.
\]
We set $p_g(D)=0$ when $D=0$.

\end{defn}

Note that if $D\in \cE^+$ and $\supp(D)\ne \emptyset$ can be contracted to normal surface singularities, then $p_g(D)$ is exactly the sum of the geometric genus of those singularities.

The next lemma follows from the arguments in  \cite[3.6]{Ok2} and \cite[\S 3]{OWY2}.

\begin{lem}
\label{l:nD}
Assume that a divisor $L$ on $X$ has no fixed components.
Then we have the following.
\begin{enumerate}
\item $h^1(\cO_X(sL))\ge h^1(\cO_X((s+1)L))$ for $s\in \Z_{\ge 0}$.
 Let 
\[
s_0=\min\defset{s\in \Z_{\ge 0}}{h^1(\cO_X(sL))= h^1(\cO_X(tL)) \; \text{for } t \ge s}.
\]

\item If $\cO_X(L)$ is generated $($by global sections$)$, then 
\[
s_0=\min\defset{s\in \Z_{\ge 0}}{h^1(\cO_X(sL))= h^1(\cO_X((s+1)L))}
\le p_g(A).
\]
\item $h^1(\cO_X(sL))=p_g(L^{\bot})$ for $s \ge s_0$.
\item $\cO_X(sL)$ is generated for $s \ge s_0 + 1$.
\end{enumerate}
\end{lem}

\section{A vanishing theorem on a partial resolution and its applications}\label{s:vanish}

The aim of this section is to prove a theorem that implies \thmref{t:Main0}.
We use the notation introduced in the preceding section.

Let us fix a divisor $L$ on $X$ that is not numerically trivial,
 and assume that $\cO_X(L)$ has no fixed components.
In particular, $L$ is nef and a general element of $H^0(\cO_X(L))$ induces an isomorphism $\cO_X\cong \cO_X(L)$ on a neighborhood of $L^{\bot}$ (cf. \eqref{eq:trivialW}).
By \lemref{l:nD} (4), there exists $n_0\in \Z_{\ge 0}$ such that $\cO_X(nL)$ is generated for every $n> n_0$.
Let $Y=\proj \bigoplus_{n\ge 0}\pi_*\cO_X(nL)$,
 and let $f\: X\to Y$ be the natural morphism.
Then $Y$ is normal, and the exceptional set of $f$ coincides with $L^{\bot}$.
Since $\cO_X(L)\cong \cO_X$ on a neighborhood of $L^{\bot}$ and the local equations of $L$ near $L^{\bot}$ are also those of $f_*L$ near $f(L^{\bot})$, we see that  $f_*L$ is a Cartier divisor on $Y$ and $\cO_Y(f_*L)$ has no base points in  $f(L^{\bot})$.


We also call a divisor $D$ on $Y$ a {\em cycle} if $\supp (D) \subset f(E)$.
Similar to the treatment of cycles on $X$, we regard a cycle $D$ on $Y$ as a scheme with structure sheaf $\cO_D := \cO_Y/\cO_Y(-D)$. We write $\chi(D)=\chi(\cO_D)$.

In general, a cycle $D$ on $Y$ is not a Cartier divisor (nor even $\Q$-Cartier), so we cannot define $f^*D$ in the usual way. 

\begin{defn}
Let $W>0$ be a cycle on $Y$.
We denote by $f_*^{-1}W$ the proper transform of $W$ on $X$.
Then $f_*\cO_X(-f_*^{-1}W)=\cO_Y(-W)$ by the normality of $Y$.
Let $F$ denote the $f$-fixed part of $\cO_X(-f_*^{-1}W)$, namely, $F$ is the maximal cycle such that 
\[
f_*\cO_X(-f_*^{-1}W-F)=f_*\cO_X(-f_*^{-1}W).
\]
Then we define a cycle $f^+W$ on $X$ by 
\[
f^{+}W=f_*^{-1}W+F.
\]
Note that $f^+W$ is anti-nef on $L^{\bot}$ and that $\supp(F)$ is a union of some connected components of $L^{\bot}$.
\end{defn}

For any coherent sheaf $\cF$ on $X$, the Leray spectral sequence implies the equality
\begin{equation}\label{eq:sp}
h^1(X, \cF)=h^1(Y, f_*\cF)+h^0(Y, R^1f_*\cF).
\end{equation}

\begin{lem}\label{l:chif+}
We have the following.
\begin{enumerate}
\item $f^*f_*L=L$.
\item $h^1(\cO_X(L))=h^1(\cO_Y(f_*L))+p_g(L^{\bot})$.
\item For a cycle $W>0$ on $Y$, we have
\[
\chi(W)-\chi(f^+W)=p_g(L^{\bot})-h^0(R^1f_*\cO_X(-f^+W))\ge 0.
\]
In particular, $\chi(W)=\chi(f^+W)$ if $p_g(L^{\bot})=0$.
\end{enumerate}
\end{lem}
\begin{proof}
(1) Let $F= f^*f_*L-L$. Then $\supp(F)$ is $f$-exceptional and $0 \ge F^2=(f^*f_*L-L)F=-LF=0$.  Hence $F=0$.

(2) It follows from \eqref{eq:sp} because we have $h^0(R^1f_*\cO_X(L))=p_g(L^{\bot})$ since $\cO_X\cong \cO_X(L)$ on a neighborhood of the $f$-exceptional set $L^{\bot}$.

(3) 
Let $W'=f^+W$.
From the exact sequence 
\[
0 \to \cO_X(-W') \to \cO_X \to \cO_{W'} \to 0,
\]
we have 
\[
\chi(W')  = \dim_k H^0(\cO_X)/H^0(\cO_X(-W')) + h^1(\cO_X(-W'))-h^1(\cO_X).
\]
We also have the similar expression of $\chi(W)$.
Since $f_*\cO_X(-W')=\cO_Y(-W)$ by the definition of $f^+$ and $f_*\cO_X=\cO_Y$, it follows from the equality \eqref{eq:sp} that 
\begin{gather*}
h^1(\cO_X(-W'))=h^1(\cO_Y(-W))+h^0(R^1f_*\cO_X(-W')), \\
h^1(\cO_X)=h^1(\cO_Y) + h^0(R^1f_*\cO_X).
\end{gather*}
Since $h^0(R^1f_*\cO_X)=p_g(L^{\bot})$, we obtain 
\begin{align*}
\chi(W') 
& = \dim_k H^0(\cO_Y)/H^0(\cO_Y(-W)) + h^1(\cO_Y(-W))-h^1(\cO_Y) \\
 & \quad \quad 
+ h^0(R^1f_*\cO_X(-W')) - h^0(R^1f_*\cO_X) \\
& = \chi(W) +  h^0(R^1f_*\cO_X(-W')) - p_g(L^{\bot}).
\end{align*}
Applying \lemref{l:nD} (1) with $s=0$ and $L=-W'$ to each singular point of $Y$, i.e., each point of $f(L^{\bot})$, we obtain that $p_g(L^{\bot}) \ge h^0(R^1f_*\cO_X(-W'))$.
\end{proof}

We use the terminology and results from \cite{CFHR}.
Let $W$ be a cycle on $Y$.
Then $\cO_W$ is torsion free since $\cO_Y(-W)$ is reflexive.
Let $\omega_W$ denote the dualizing sheaf of $W$.
The {\em degree} of a torsion free sheaf $\cL$ on $W$ is defined to be
$\deg \cL = \chi(\cL)-\chi(W)$.
For a Cartier divisor $D$ on $Y$, let $\cO_W(D)=\cO_Y(D)\otimes \cO_W$ and 
 $DW:=\deg \cO_W(D)$.

\begin{lem}\label{l:nv}
Let $D$ be a Cartier divisor on $Y$ and let $W>0$ be a cycle on $Y$ such that $H^1(\cO_W(D))\ne 0$.
Assume that $W$ is minimal among all cycles on $Y$ with this property.
Then $DW \le -2\chi(W)$.
\end{lem}
\begin{proof}
By duality, we have $\Hom (\cO_W(D), \omega_W)\ne 0$. 
Let $0\ne \phi \in \Hom (\cO_W(D), \omega_W)$.
Since $\cO_W(D)$ is invertible, 
it follows from Lemma 2.4 and Remark 2.5 of \cite{CFHR} that 
if $W'\subset W$ is a subscheme defined by the ideal $\ann \phi \subset \cO_W$, then $W'$ is a cycle and $\phi$ factors as $\cO_W(D) \to \cO_{W'}(D) \xrightarrow{\; \phi'} \omega_{W'} \hookrightarrow \omega_W$,
where $\phi'$ is generically surjective.
However, by duality, the existence of such $\phi'$ implies that $H^1(\cO_{W'}(D))\ne 0$.  By the minimality of $W$, we obtain that $W=W'$.
Therefore, the morphism $\phi\: \cO_W(D) \to \omega_W$ is generically surjective and, since $\omega_W$ is invertible at general points, also generically injective.
Since $\cO_W(D)$ is torsion free, $\phi$ is in fact injective.
Hence, we have the exact sequence
\[
0 \to \cO_W(D) \xrightarrow{\phi} \omega_W \to \Coker \phi \to 0,
\]
and that $\supp(\Coker \phi)$ is a finite set.
Thus we obtain that 
\[
\chi(\omega_W) - \chi(\cO_W(D)) = h^0(\Coker \phi) \ge 0.
\]
Since $DW = \chi(\cO_W(D)) - \chi(W)$ and $\chi(\omega_W) = - \chi(W)$ by duality, we have 
\[
DW = \chi(\cO_W(D)) - \chi(W) \le \chi(\omega_W)-\chi(W) = -2 \chi(W).
\qedhere
\]
\end{proof}

Let $\cB(L)^e$ be as in \sref{s:R}.
The main result in this section is the following.

\begin{thm}\label{t:Main}
Assume that $\cO_X(L)$ has no fixed components and that 
 $LC> -2\chi(C)$ for all $C\in \cB(L)^e$.
Then $H^1(\cO_Y(f_*L))=0$ and $h^1(\cO_X(L))=p_g(L^{\bot})$.
\end{thm}
\begin{proof}
By \lemref{l:chif+} (2), it suffices to prove that $H^1(\cO_Y(f_*L))=0$.
Recall that $f_*L$ is a Cartier divisor and $\cO_Y(f_*L)$ has no base points in $f(L^{\bot})$.
Let $W>0$ be an arbitrary cycle on $Y$.
Since $f^{+}W \in \cE^+(L)$ by the definition of $f\: X\to Y$, it follows from  \proref{p:Rohr} that $Lf^{+}W>-2\chi(f^{+}W)$. 
Note that if $0\ne \sigma\in H^0(\cO_W(f_*L))$ is a general section, 
then the morphism $\sigma: \cO_W \to \cO_W(f_*L)$ is injective and $\supp(\Coker \sigma)$ is a finite subset of $f(E)\setminus f(L^{\bot})$.
Hence we have $(f_*L)W=h^0(\Coker \sigma)=(f^*f_*L)f^{-1}_*W$.
Using \lemref{l:chif+}, we obtain
\[
(f_*L)W=(f^*f_*L)f_*^{-1}W=(f^*f_*L)f^{+}W
=Lf^{+}W>-2\chi(f^{+}W) \ge -2\chi(W).
\]
Therefore, by \lemref{l:nv}, $H^1(\cO_W(f_*L))=0$ for an arbitrary cycle $W>0$ on $Y$.
Hence we have $H^1(\cO_Y(f_*L))=0$.
\end{proof}

\begin{rem}\label{r:chiL}
 The condition 
\begin{equation}\label{eq:condition}
\text{ $LC> -2\chi(C)$ for all $C\in \cB(L)^e$ }
\end{equation}
in \thmref{t:Main} may be replaced by conditions that use $\cB(L)$ instead of $\cB(L)^e$, as follows.
Let 
\[
\chi_L=\min\defset{\chi(D)}{D\in \cE^+, \;\supp (D) \subset L^{\bot}}
\]
 and suppose  $C\in \cB(L)^e$.  
By the definition of $\cB(L)^e$, $C$ can be expressed as $C=C_1+C_2$, where $C_1\in \cB(L)$, $\supp(C_2)\subset L^{\bot}$, and $C_1C_2\le 0$.
We have $LC=LC_1$, and by \eqref{eq:chi} and \eqref{eq: h1B},
\[
\chi(C) \ge \chi(C_1)+\chi(C_2) \ge \chi(C_1)+\chi_{L} \ge \chi(Z_f)+\chi_L,
\] 
where $Z_f$ is the fundamental cycle.
Therefore, the condition \eqref{eq:condition} can be replaced by 
\begin{enumerate}
\item $LC> -2(\chi(C)+\chi_L)$ for all $C\in \cB(L)$, or 
\item $LC> -2(\chi(Z_f)+\chi_L)$ for all $C\in \cB(L)$.
\end{enumerate}
If $L^{\bot}$ contracts to rational or elliptic singularities, 
then $\chi_L\ge 0$,  and thus $\chi_L$ can be omitted in the conditions (1) and (2) above.
\end{rem}

The next example shows that the assumption in \thmref{t:Main} requiring $\cO_X(L)$ to have no fixed components is indispensable. 

\begin{ex}
Let us recall an example in \cite[6.3]{Ok1}.
Let $(S,o)\subset (\C^4,o)$ be a normal complex surface singularity defined by the $2\times 2$ minors of the matrix
 \[
\begin{pmatrix}
x&y&z \\ y-3w^2 & z+w^3 & x^2+6wy-2w^3
\end{pmatrix}.
\]
This is an elliptic singularity with $p_g(S,o)=1$.
Assume that $X\to S$ is the minimal resolution with exceptional set $E$.
Then $E$ consists of a rational curve $E_1$ with $(2,3)$-cusp and a nonsingular rational curve $E_2$ such that $E_1^2=-1$, $E_2^2=-2$, and $E_1E_2=1$.
In fact, the resolution graph of $(S,o)$ coincides with that of the hypersurface singularity $(\{x^2+y^3+z^{13}=0\}, o) \subset (\C^3,o)$. 
We have $Z_f=E$, and the maximal ideal cycle (e.g., \cite[\S 3]{Ok1}) is $Z_f+E_1$,  because $(S,o)$ is of type II in \cite[3.11]{Ok1} with $\alpha=\beta=1$.
Hence, the fixed part of $\cO_X(-Z_f)$ is just $E_1$.
Let $L=-nZ_f$ for $n\in \Z_{>0}$.
Clearly, $L^{\bot}=Z_f^{\bot}=E_1$, and $p_g(L^{\bot})=1$ since $E_1$ contracts to a minimally elliptic singularity (\cite{la.me}).
Every $C\in \cB(L)^e$ can be expressed as $C=mE_1 + E_2$ ($m\in \Z_{\ge 0}$).
Thus, $\chi(C)=\chi(mE_1)+\chi(E_2)-mE_1E_2 = (m-1)(m-2)/2 \ge 0$, and the condition \eqref{eq:condition} is satisfied.
We will show that $h^1(\cO_X(L))=0$ for every $n\in \Z$, $n\ne 0$; that is, $h^1(\cO_X(L)) \ne p_g(L^{\bot})$ for $L=-nZ_f$.
We have $\cO_{E_1}(-Z_f)\not \cong \cO_{E_1}$ by \cite[2.13, 2.15]{Ok1}.
Since $\pic(E_1)$ has no torsion elements (see the paragraph preceding \cite[3.7]{Ok1}), it follows that $\cO_{E_1}(-n Z_f)\not \cong \cO_{E_1}$ for every $n\in \Z$, $n\ne 0$.
However, since $\chi(\cO_{E_1}(-nZ_f))=\chi(\cO_{E_1})=0$, we have $h^0(\cO_{E_1}(-nZ_f))=h^1(\cO_{E_1}(-nZ_f))$.
It is well-known that a numerically trivial line bundle $\cL$  on a reduced cycle is analytically trivial if $H^0(\cL)\ne 0$.
Hence $h^1(\cO_X(-nZ_f))=0$ for every $n\in \Z$, $n\ne 0$.
\end{ex}

\section{An upper bound of the normal reduction number}\label{s:nr}

We apply the results from \sref{s:vanish} to study the normal reduction number of integrally closed $\m$-primary ideals of the surface singularity $A$.  
We follow the notation introduced in the preceding sections.

Let $I$ be an $\m$-primary ideal of $A$.  
The integral closure $\bar{I} \subset A$ of $I$ is the ideal consisting of all solutions $z$ of an equation of the form
 $z ^n +c_1 z^{n-1} +c_2 z^{n-2} + \dots+ c_{n-1} z + c_n=0$ with coefficients $c_i \in I^i $. 
Then $ I \subseteq \bar{I} \subseteq \m$. 
We say that $I$ is \textit{integrally closed} if $I = \bar{I}$.

\begin{defn}[{\cite{OWY4}}]
Let $I \subset A$ be an integrally closed $\m$-primary ideal, 
and let $Q$ be a minimal reduction of $I$. 
The {\em normal reduction number }$\br(I)$ of $I$ is defined by
\[
\br(I) = \min\{ r \in \mathbb{Z}_{>0} \;|\; \overline{I^{n+1}} = Q \overline{I^n} \; \mbox{\rm for all } n \ge r\}.
\]
This is independent of the choice of $Q$.
The normal reduction number $\br(A)$  of the singularity $A$ is defined by 
\[
\br(A)=\max\defset{\br(I)}{\text{ $I$ is an integrally closed $\m$-primary  ideal of $A$}}.
\]
\end{defn}

\begin{prop}[{\cite{Ok2}, \cite{OWY5}}]
\label{p:reh}
We have  the following.
\begin{enumerate}
\item $A$ is rational if and only if $\br(A)=1$.
\item If $A$ is elliptic, then $\br(A)=2$.
\item If $A=k[[x,y,z]]/(f)$, where $f$ is a homogeneous polynomial of degree $d$, then $\br(A)=\br(\m)=d-1$.
\end{enumerate}
\end{prop}

Let $I\subset A$ be an integrally closed $\m$-primary ideal.
Then there exist a resolution $f \colon X \to \Spec A$ and an anti-nef cycle $Z>0$ on $X$ so that 
$I = H^0(\mathcal{O}_X(-Z))$  and $I\mathcal{O}_X=\mathcal{O}_X(-Z)$;
 in this case, we say that $I$ is {\em represented} by $Z $ on $X$ and write $I = I_Z$.   

{\bf In the following, $I$ always denotes an integrally closed $\m$-primary ideal of $A$ represented by a cycle $Z>0$ on $X$.  
}  Then the sheaf $\cO_X(-Z)$ has all information about $I$.

\begin{defn} \label{qI} 
For every integer $n \in \Z_{\ge 0}$, we define $q_I(n):= h^1(\mathcal{O}_X(-nZ))$. 
In particular, $q_I(0)=p_g(A)$. 
Notice that $q_I(n)$ is independent of the representation (\cite[Lemma 3.4]{OWY1}). 
\end{defn}

By \lemref{l:nD}, we have $q_I(n)\ge q_I(n+1)$ for every $n\in \Z_{\ge 0}$,
 and $q_I(n) =p_g(Z^{\bot})$ for every $n\ge p_g(A)$.
If $q_I(1)=q_I(0)=p_g(A)$, we have $q_I(n)=p_g(A)$ for $n\in \Z_{\ge 0}$; such an ideal is called a $p_g$-ideal (\cite{OWY1}).

\begin{prop}[{\cite[\S 2]{OWY4}}]
\label{p:brq}
We have the following.
\begin{align*}
\br(I) & = \min\{n \in \Z_{>0} \,|\, q_I(n-1)=q_I(n)  \} \\
& = \min\{n \in \Z_{>0} \,|\, q_I(n-1)=p_g(Z^{\bot})  \}.
\end{align*}
\end{prop}

\lemref{l:nD} and \proref{p:brq} imply the following.

\begin{cor}
\label{c:r=2}
We have the following.
\begin{enumerate}
\item $\br(I)=1$ if and only if $p_g(A)=q_I(1)$ or $p_g(Z^{\bot})$.
\item $\br(I)=2$ if and only if $p_g(A)>q_I(1)=p_g(Z^{\bot})$.

\item  $\br (A) \le p_g(A)+1$.
\end{enumerate}
\end{cor}

Next, we will give 
sharpened upper bounds for $\br(A)$ compared with \corref{c:r=2} (3).

\begin{defn}\label{d:lambda}
Let $U\to \spec (A)$ be a resolution and let  $W>0$ be a cycle on $U$ representing an integrally closed $\m$-primary ideal of $A$.
We define 
\[
\lambda(W,U)=\max\defset{\fl{ \frac{2p_a(D)-2}{-WD} } }{D\in \cB(W)^e \vphantom{\dfrac{A}{B}}},
\]
where $\cB(W)^e$ is defined for cycles on $U$ in the same way as for cycles on $X$. 
We define invariants $\lambda(I)$ and $\lambda(A)$ as follows:
\begin{gather*}
\lambda(I)=\min\defset{\lambda(W,U)}{\text{$U$ is a resolution and $W$ a cycle on $U$ representing $I$}} \\
\lambda(A)=\max\defset{\lambda(I)}{\text{ $I$ is an integrally closed $\m$-primary  ideal of $A$}}.
\end{gather*}
\end{defn}

\begin{thm}
\label{t:lambda}
We have the following inequalities:
\begin{enumerate}
\item $\br(I) \le \lambda(I)+2 \le \lambda(Z,X) + 2$,
\item $\br(A) \le \lambda(A) +2$.
\end{enumerate}
\end{thm}
\begin{proof}
Let $\lambda=\lambda(Z,X)$. 
For any  $D\in \cB(Z)^e$, since $-ZD>0$, we have 
\[
(\lambda+1)\cdot (-ZD) 
\ge \left(\fl{ \frac{-2\chi(D)}{-ZD} } +1\right)\cdot (-ZD) > -2\chi(D).
\]
Therefore, by \thmref{t:Main}, $h^1\left(\cO_X(-(\lambda+1) Z)\right)=p_g(Z^{\bot})$.
 Hence, $\br(I) \le \lambda +2$ by \proref{p:brq}.
Since the same argument applies to every representation of $I$, 
the inequality (1) follows.
The inequality (2) follows directly from the definitions of $\lambda(A)$ and $\br(A)$.
\end{proof}

\begin{rem}
If $-ZD\ge 2$ for any $D\in \cB(Z)^e$, then 
\[
\lambda(Z,X) \le \max\defset{p_a(D) -1 }{D\in \cB(Z)^e} \le p_a(A)-1.
\]
Hence we easily obtain $\br (I) \le p_a(A)+1$ in this case.
\end{rem}

The following theorem gives a  combinatorial upper bound for $\br(A)$.

\begin{thm}\label{t:brp}
$\br(A)\le p_a(A)+1$.
\end{thm}
\begin{proof}
We may assume that $p_a(A)>0$, i.e., $A$ is not a rational singularity.
Let $I$ be any integrally closed $\m$-primary ideal and assume that $I$ is represented by a cycle $Z$ on $X$.
We will show that the inequality $\br(I)\le p_a(A)+1$ holds.
By \proref{p:brq}, it suffices  to show that $H^1(\cO_X(-p_a(A)Z))=p_g(Z^{\bot})$.
We apply the argument of \sref{s:vanish} with $L=- p_a(A) Z$. 
Let $f\: X\to Y$ be the morphism that contracts $Z^{\bot}$ as in \sref{s:vanish}.
Note that $\cO_Y(-f_*Z)$ is also invertible,  generated, and $(-f_*Z)W>0$ for every cycle $W>0$ on $Y$.
 Let $p=p_a(A)$ and $D=-pf_*Z$. 
By \lemref{l:chif+} (2), it is sufficient to show that $h^1(\cO_Y(D))=0$.
 Suppose that $H^1(\cO_Y(D))\ne 0$.
Then there exists a cycle $W>0$ on $Y$ such that $H^1(\cO_W(D))\ne 0$.
Assume that $W$ is the minimal cycles with this property.
Then $0<DW \le -2\chi(W)$ by \lemref{l:nv}.
From \lemref{l:chif+} (3) and the definition of $p_a(A)$, 
we have $\chi(W) \ge \chi(f^+W)\ge 1-p$, and thus $DW \le 2(p-1)$.
It follows that $0<(-f_*Z)W \le 2(1-1/p)<2$.
Therefore,  $(-f_*Z)W=1$ and $W$ is irreducible.
However, since $\cO_W(-f_*Z)$ is generated and $p_a(W)=1-\chi(W)>1$, we must  have $\deg \cO_W(-f_*Z)=(-f_*Z)W \ge 2$, a contradiction.
Hence, we obtain $h^1(\cO_Y(D))=0$.
\end{proof}

Note that \thmref{t:brp} offers a new proof for (2) of \proref{p:reh}.

\section{almost cone singularities}\label{s:ACS}
We continue to use the notation of the preceding section.
We say that $A$ is a {\em cone-like singularity} if the exceptional set of of the minimal resolution is a smooth curve.
In \cite{OWY5}, we proved that for a non-rational cone-like singularity with the exceptional curve $E$ on the minimal resolution,  
\[
\br(A) \le \fl{(2g-2)/\min\{-E^2, \gon(E)\}}+2,
\]
 where $\gon(E)$ denotes the gonality of $E$. 
The aim of this section is to generalize and sharpen this result.
The following results on the minimal model of a cycle will be used for introducing the notion of almost cone singularities, 
which generalize the cone-like singularities, 
and for the proof of facts in \exref{e:br346}.
\begin{prop}
[Konno {\cite[3.1, 3.3]{Ko-CC}}]
\label{p:minmodel}
For $D\in \cB$ with $\chi(D)\le 0$, there exists a unique cycle $0< \mc D \le D$ such that 
$\chi(\mc D)=\chi(D)$ and that $K_X+\mc D$ is nef on $\mc D$.
Moreover, we also have $\mc D \in \cB$ and the following:
\begin{align*}
\mc D &=\min\defset{C\in \cE^+}{C\le D, \, \chi(C)=\chi(D)} \\
&=\max\defset{C\in \cE^+}{C\le D, \, K_X+C \text{ is nef on $C$}}.
\end{align*}
\end{prop}

\begin{defn}
We call $\mc D$ the {\em minimal model} of $D$. 
\end{defn}

\begin{lem}
\label{l:outmc}
Let $D\in \cB$ and let $\{C_1, \dots, C_m\}\subset \cB$ be a computation sequence from $C_1$ to $C_m=D$ such that $\mc{D}\le C_1$ and $C_i=C_{i-1}+E_{j_i}$ $(1<i\le m)$.
Then we have the following$:$
\[
\chi(C_1)=\chi(C_i), \ \ p_a(E_{j_i})=0, \ \ C_{i-1} E_{j_i}=1 \ \ (1<i\le m).
\]
\end{lem}
\begin{proof}
Since $\chi(C)=1-h^1(\cO_C)$ for every $C\in \cB$ by \eqref{eq: h1B}, 
from the surjections $\cO_{C_{i+1}} \to \cO_{C_i} \to \cO_{\mc D}$,
we have $\chi(\mc D) \ge \chi(C_i) \ge \chi(C_{i+1})$. 
The equality $\chi(\mc D)=\chi(D)$ gives $\chi(C_1)=\chi(C_i)$ for every $i$.
From \eqref{eq:chi}, we have $\chi(C_{i})=\chi(C_{i-1})+\chi(E_{j_i})- C_{i-1} E_{j_i}$, and thus $\chi(E_{j_i})= C_{i-1} E_{j_i}>0$.
This implies that $p_a(E_{j_i})=h^1(\cO_{E_{j_i}})=0$ and $C_{i-1} E_{j_i}=1$.
\end{proof}

\begin{defn}
Let $W>0$ be a cycle on $X$. We denote by $\red{W}$ the reduced cycle such that $\supp(\red{W})=\supp(W)$.
Assume that $\red{W}<E$. We call an irreducible component $E_i$ of $W$ a {\em connecting component} if $(E-\red{W})E_i>0$.
\end{defn}

\begin{prop}\label{p:outCZ_B}
Assume that $\red{(\mc{(Z_f)})} < E$.
Let $B$ be a connected component of $E-\red{(\mc{(Z_f)})}$ and $Z_B$ the fundamental cycle on $B$.
Then $\chi(Z_B)=1$, $\mc{(Z_f)}+Z_B \in \cB$, and $Z_B \mc{(Z_f)}=1$.
\end{prop}
\begin{proof}
Let $C_1=\mc{(Z_f)}$. 
There exists a sequence $\{C_1, \dots, C_m\}\subset \cB$ such that $C_i=C_{i-1}+E_{j_i}$, $E_{j_i}\le B$ $(1<i\le m)$, and $C_m-C_1$ is anti-nef on $B$.
Moreover, assume that $C_m$ is minimal with this property.
Clearly, $\mc{(C_i)}=C_1$ for $1\le i \le m$.
By \lemref{l:outmc}, we have $C_1B=1$.
Hence $B_1:=E_{j_2}$ is the unique connecting component of $B$.
If $(C_i-C_1)B_1>0$ for some $i\ge 2$, then $C_iB_1>C_1B_1=1$; this is impossible by \lemref{l:outmc}.
Therefore, $(C_i-C_1)B_1\le 0$ for every $i\ge 2$.
This shows that $C_m-C_1$ is reduced at $B_1$ and 
 the sequence $\{C_2-C_1=B_1, \dots, C_m-C_1\}$ is a computation sequence for $Z_B$
by the assumption on the sequence $\{C_i\}$.
Hence $C_1+Z_B\in \cB$ and 
 $C_1Z_B=C_1B_1=1$.
By \lemref{l:outmc} and \eqref{eq:chi}, we have 
\[
\chi(C_1)=\chi(C_1+Z_B)=\chi(C_1)+\chi(Z_B)-1.
\]
Hence $\chi(Z_B)=1$.
\end{proof}

\begin{lem}
\label{l:minpb}
Let $f\: Y \to X$ be the blowing up at a point in $E$ with exceptional curve $E_0\subset Y$.
For any $D\in \cE^+$ on $X$, we have the following.
\begin{enumerate}
\item $D$ is chain-connected if and only if so is $f^*D$.
\item 
Assume that $D$ is chain-connected and $\chi(D) \le 0$.
If 
 $f(E_0)\subset \supp(\mc{D})$, then 
\[
\mc{(f^*D)}=f^*(\mc D) -E_0.
\]
\end{enumerate}
\end{lem}
\begin{proof}
(1) If there exist cycles $D_1, D_2$ on $X$ such that 
$D=D_1+D_2$, $0<D_1<D$, and $D_1$ is anti-nef on $D_2$, then clearly  $f^*D_1$ is anti-nef on $f^*D_2$. 
Therefore,  if $f^*D$ is chain-connected,  so is $D$.
Next, assume that $f^*D$ is not  chain-connected; that is, 
there exist cycles $\ol D_1, \ol D_2$ on $Y$ such that 
$f^*D=\ol D_1+\ol D_2$, $0<\ol D_1<f^* D$, and $\ol D_1$ is anti-nef on $\ol D_2>0$.
Let $D_i=f_*\ol D_i$. 
Since $D=D_1+D_2$, there exists $m\in \Z$ such that $\ol D_1 = f^*D_1+mE_0$ and $\ol D_2 = f^* D_2 - m E_0$. 
If $m<0$, then $E_0 \le \ol D_2$ and $\ol D_1 E_0 = -m>0$; 
it contradicts the assumption that $\ol D_1$ is anti-nef on $\ol D_2$.
Hence $m\ge 0$, and thus $D_2>0$.
For any irreducible component $C \le D_2$, 
letting $f_*^{-1}C$ the proper transform of $C$, we have 
\[
0 \ge \ol D_1 f_*^{-1}C =(f^*D_1)f_*^{-1}C+ mE_0f_*^{-1}C \ge  (f^*D_1)f_*^{-1}C = D_1C. 
\]
Therefore, $D_1$ is anti-nef on $D_2$.
Hence $D$ is  not chain-connected.

(2) 
Let $D'=f^*(\mc D) - E_0$.
Then $0< D' < f^*D$ since $f(E_0)\subset \supp(\mc D)$, and the cycle $K_Y +D' = f^*(K_X+\mc D)$ is nef on $D'$.
Hence the claim follows from the equalities
\[
\chi(D')= -\frac{1}{2}(f^*\mc D) f^*(K_X+\mc D) = \chi(\mc D)
=\chi(D)=\chi(f^*D).
\qedhere
\]
\end{proof}

To define almost cone singularities, let us consider a singularity with $p_f(A)\ge 1$ and the following condition:

\begin{enumerate}
\item[(AC)]  There exists a nonsingular irreducible curve $C \le E$ such that $p_a(C)=p_f(A)$ $($that is, $C=\mc{(Z_f)}$$)$ and $Z_f C<0$.
\end{enumerate}

\begin{lem}\label{l:chi=1}
Assume that  $p_f(A)\ge 1$ and  the condition {\rm (AC)} is satisfied.
Let $D\in \cB$.
\begin{enumerate}
\item If $C\not\le D$, then $\chi(D)=1$ and $DC\le 1$.
In particular, each connected component of $E-C$ is the exceptional set of a rational singularity and the weighted dual graph of $E$ is a tree.

\item If $C\le D$, then $\chi(D)=\chi(C)$ and $D$ is reduced at $C$, namely, the coefficient of $C$ in $D$ is one.
In particular, $Z_f$ is reduced at $C$.
\end{enumerate}
\end{lem}
\begin{proof}
(1) 
There exists a connected component $B$ of $E-C$ such that $\red{D}\le B$.
Hence the assertion follows from \proref{p:outCZ_B}.

(2) Let  $Z_1, \dots , Z_{\ell}$ be a computation sequence for $Z_f$.
Assume that  $C\le Z_{m-1}$ and $Z_m = Z_{m-1} + C$.
Using \eqref{eq:chi}, we have 
\begin{equation}\label{eq:Cappear}
\chi(Z_m)=\chi(Z_{m-1})+\chi(C)-Z_{m-1}C.
\end{equation}
Since $\chi(Z_i)=\chi(C)$ for $i\ge m-1$ by \lemref{l:outmc}, using \eqref{eq:Cappear}, we have $\chi(C)=Z_{m-1}C>0$.
It contradicts that $\chi(C)=1-p_a(C)=1-p_f(A)\le 0$.
\end{proof}

\begin{lem}
\label{l:forAC}
Assume that $p_f(A)\ge 1$ and the condition {\rm (AC)} is satisfied.
Then the condition {\rm (AC)} holds for every resolution of $\spec(A)$.
Moreover, the intersection number $Z_f C$ is independent of the resolution.
\end{lem}
\begin{proof}
Let us consider the situation of \lemref{l:minpb}.
Then $f^*Z_f$ is the fundamental cycle on $Y$.
We show that  the condition {\rm (AC)} is satisfied on $Y$.
If $f(E_0)\cap C=\emptyset$, the assertion is clear.
Assume that $f(E_0)\subset C$. Then we obtain
\[
\mc{(f^*Z_f)}=f^*C-E_0=f^{-1}_*C \ \text{and} \ 
(f^*Z_f)(f^{-1}_*C) = Z_f C <0.
\]
Hence {\rm (AC)} holds on $Y$.

Next, assume that there exists a $(-1)$-curve $F$ on $X$ and let $\phi\: X\to X'$ be the contraction of $F$. 
Let $D$ be the fundamental cycle on the connected component of $E-C$ including $F$.
By \lemref{l:chi=1} (1), $1\ge DC \ge FC$.
This implies that $\phi(C)$ is nonsingular.
We easily see that $\phi(C)$ is the minimal model of the fundamental cycle $\phi_*Z_f$ on $X'$, 
and $(\phi_*Z_f)\phi(C)=(\phi^*\phi_*Z_f)(\phi^*\phi(C))=(\phi^*\phi_*Z_f)C=Z_fC$.
\end{proof}

\begin{defn}\label{d:AC}
Assume that $p_f(A)\ge 1$, and let $X \to \spec (A)$ be any resolution.
We say that $A$ is an {\em almost cone  singularity} if there exists a nonsingular irreducible curve $C \le E$ such that $p_a(C)=p_f(A)$ and $Z_f C<0$.
In this case, we call $C$ the {\em central curve} of $E$ and the positive integer $-Z_fC$ the {\em degree} of the almost cone singularity $A$.
\end{defn}

\begin{ex}
Assume that $C\subset E$ is a nonsingular curve with $p_a(C) \ge 1$.

(1)
 Let $C_1, \dots, C_m$ be the connected components of $E-C$ and $Z_i$ the fundamental cycle on $C_i$.
Assume that  $\chi(Z_i)=1$, $Z_i(E-C_i)=Z_iC=1$ ($i=1, \dots, m$), $C^2<-m$, and $Z_f=C+Z_1+\cdots +Z_m$.
Then $A$ is an almost cone  singularity and $C$ is the central curve.

(2)
 Assume that $E$ is star-shaped (cf. \cite{tki-w}) and  the central curve $C$ of $E$ is not rational.
Then $A$ is an almost cone singularity if and only if $EC<0$; if this is the case, $Z_f=E$.

(3)
 Any graded almost cone singularity is obtained as follows.
Let $D$ be a $\Q$-divisor on $C$ such that the degree of the integral part $\fl{D}$ is positive. Then the graded ring $R(C,D)=\bigoplus_{n\ge 0} H^0(\cO_C(\fl{nD}))$ has an almost cone  singularity of degree $\deg\fl{D}$ at the homogeneous maximal ideal of $R(C,D)$ (see \cite[\S 6]{tki-w} for the resolution graph of $R(C,D)$).
\end{ex}

We further investigate the combinatorial properties of $E$ with respect to the upper bound of the normal reduction number of almost cone singularities.

\begin{lem}
\label{l:connecting}
Let $B$ be a cycle with connected support such that $0< B <E$ and let $Z_B$ be the fundamental cycle on $B$.
Let $M=\mc{(Z_f)}$ and assume that $M=\mc{(Z_B)}$.
\begin{enumerate}
\item If $E_i\le E-B$ and $BE_i>0$, then $Z_B E_i=1$.
In particular, $Z_B$ is reduced at every connected component of $Z_B$.

\item Assume that $Z_fM<0$ and  the effective cycles $Z_f-M$ and $Z_B-M$ contain no components of $M$.
  Then, there exists a cycle $W$ such that $\red{W}=B$, $W$ is anti-nef on $B$ and reduced at every connecting component, and satisfies $WM \le -2$.
\end{enumerate}
\end{lem}
\begin{proof}
(1) 
We have $Z_B+E_i\in \cB$ since $Z_BE_i>0$.
By \lemref{l:outmc}, 
 we have  $Z_BE_i =1$.

(2) Since $Z_f-Z_B\ge 0$ and it has no components of $M$, we have $Z_B M \le Z_f M \le -1$.
If $Z_B M\le -2$, then put $W=Z_B$.
Assume that $Z_B M = Z_f M = -1$. The equality $Z_B M = Z_f M$ yields $\supp(Z_f-Z_B) \cap \supp(M)=\emptyset$.
It also follows that $B>\red M$.
Indeed, if $B=\red M$, we have $Z_B=M$, and thus $(Z_f-Z_B)M>0$.

\begin{clm}\label{cl:1}
Let $E_i$ be any connecting component of $Z_B$ and let $L$ be the minimal reduced cycle such that $E_i\le L$ and $\red M \le L$.
Then, $L-\red M$ is a chain of curves and  there exists a component $E_j \le L - \red M$ such that $Z_B E_j<0$. 
\end{clm}
\begin{proof}
[Proof of \clmref{cl:1}]
By \proref{p:outCZ_B}, any connected component of $E-\red M$ is a tree of curves.
Hence $L-\red M$ is a chain of curves by the minimality of $L$.
Let $E_{i+1}\le E-B$ satisfies $E_{i+1}E_i>0$. 
If $E_i \le M$, then $(Z_f-Z_B)M \ge E_{i+1}M>0$, namely, $Z_B M < Z_f M$; it contradicts the assumption.
Hence $E_i\not \le M$.
Assume there is no $E_j$ as in the claim.
Then $Z_BE_i=0$ since $E_i \le L-\red M$. 
We have $Z_B E_{i+1}>0$ and $(Z_B+E_{i+1})E_i=E_{i+1}E_i>0$.
Hence $Z_B+E_{i+1}+E_i\in \cB$.
By considering a computation sequence for $Z_f$, it is easy to see that $Z_B+E_{i+1}+(L-\red M) \in \cB$; however, it contradicts that $\supp(Z_f-Z_B) \cap \supp(M)=\emptyset$, because $(E_{i+1}+(L-\red M)) M >0$.
\end{proof}

Let $B'$ be the connected component of $\supp((Z_B)^{\bot}+M)$ that contains $\supp(M)$, and let $Z_{B'}$ be the fundamental cycle on $B'$.
Since $M\in \cB$ by \proref{p:minmodel} and $B \ge B'\ge \red M$, we have $Z_B\ge Z_{B'} \ge M$.
Hence $Z_{B'}-M$ is effective and contains no components of $M$, 
and thus $Z_{B'}M \le Z_B M =-1$.
Let $W=Z_B+Z_{B'}$.
Then $WM \le -2$ and $\red W = B$.

\begin{clm}\label{cl:2}
$W$ is anti-nef on $B$.
\end{clm}
\begin{proof}
[Proof of \clmref{cl:2}]
Let $E_m$ be a component of $B$. 
If $E_m\le B'$, then $W E_m=Z_B E_m + Z_{B'}E_m \le 0$.
If $E_m \cap B'=\emptyset$, then $W E_m=Z_B E_m\le 0$.
Assume that $E_m\le B-B'$ and $B' E_m>0$. 
 Since $Z_B E_m<0$ by the definition of $B'$ and $Z_{B'}$ is also reduced at its connecting components by (1), 
 we have that $W E_m =Z_B E_m + Z_{B'}E_m \le -1 + 1 \le 0$.
Hence  $W$ is anti-nef on $B$.
\end{proof}

It follows from \clmref{cl:1} that $B'$ does not contain any connecting component of $Z_B$.
Therefore, $W$ is also reduced at its connecting components.
We have proved \lemref{l:connecting}.
\end{proof}

Recall that  the {\em gonality} of a nonsingular projective curve 
$C$ is defined as the minimum of the degree of surjective morphisms 
from $C$ to $\PP^1$. We denote the gonality of $C$ by $\gon(C)$.

The following theorem is a direct  generalization of a result for cone-like singularities in \cite{OWY5}; however, the very last inequality is proved here as a new result.

\begin{thm}\label{t:ACmain}  
Let $A$ be an almost cone singularity 
and let $I = I_Z$ be an integrally closed $\m$-primary ideal of $A$ represented by a cycle $Z$ on $X$.  
Let $C$ denote the central curve with $g:=p_a(C)=p_f(A)$ and let $d$ be the degree of $A$.
Then we have the following.
\begin{enumerate}
\item If $ZC<0$, then $\br(I)\le \fl{(2g-2)/\gon(C)}+2$.

\item If  $ZC=0$, then  $\br(I)\le \fl{(2g-2)/\delta}+2$,
where $\delta=\max\{2, d\}$.
\end{enumerate}
Hence we obtain 
\[
\br(A)\le \fl{\frac{2g-2}{\min\{\gon(C), \delta\}}}+2 \le g+1.
\]
\end{thm}

\begin{proof}
For every $D \in \cB$, by \lemref{l:chi=1}, 
we  have either
\begin{enumerate}
\item [(i)] $C\not\le D$ and $\chi(D) = 1 \ge DC$, or 
\item [(ii)] $C\le D$ and $\chi(D) =\chi(C)=1-g$.
\end{enumerate}

(1) Assume that $ZC<0$.
Since $\cO_C(-Z)$ is also generated, 
there exist sections $\sigma_1, \sigma_2\in H^0(\cO_{C}(-Z))$ which determine a surjective morphism $\phi\: C\to \PP^1$ of degree $-ZC$.
Hence $-ZC\ge \gon(C)$. 
To find an upper bound for $\lambda(Z,X)$ in \defref{d:lambda}, from  (i) and (ii) above, it is enough to take only $D\in \cB$ such that $D\ge C$ (cf. \remref{r:chiL}).
For such a cycle $D$, we have 
\[
\frac{2p_a(D)-2}{-ZD} \le \frac{2g-2}{-ZC} \le \frac{2g-2}{\gon{C}}.
\]
Hence $\lambda(Z,X) \le \fl{(2g-2)/\gon{(C)} }$, and the assertion (1) follows from \thmref{t:lambda}.

(2) Assume that $ZC=0$.
Let $B$ be the connected component of $Z^{\bot}$ containing $C$ 
and let $W$ be as in \lemref{l:connecting} (2) (note that $M=C$).
Then $-W C \ge \delta$.
The same argument as in the proof of \clmref{cl:2} shows that $Z+W$ is anti-nef.

Let  $s > (2g-2)/\delta$. 
We show that $H^1(\cO_X(-s(Z+W)))=0$ by applying R{\"o}hr's vanishing theorem (\thmref{t:rohrV}).
Let $D\in \cB$. If $C\le D$, then 
\[
-s(Z+W)D \ge 
-s(Z+W)C = -s W C > \frac{2g-2}{\delta}\cdot \delta = -2 \chi(D).
\]
If $C\not\le D$, then $\chi(D)=1$ and  $-s(Z+W)D \ge 0 > -2=-2\chi(D)$. 
Hence we obtain that $H^1(\cO_X(-s(Z+W)))=0$ by \thmref{t:rohrV}.
From the exact sequence
\[
0 \to \cO_X(-s(Z+W)) \to \cO_X(-sZ) \to \cO_{sW}(-sZ) \to 0,
\]
we have $H^1(\cO_X(-sZ))\cong H^1(\cO_{sW}(-sZ))\cong H^1(\cO_{sW})$ by \eqref{eq:trivialW}.
For sufficiently large $s'>s$, we have $h^1(\cO_X(-s'Z))=p_g(Z^{\bot})=h^1(\cO_{s'W})\ge h^1(\cO_{sW})$ by \lemref{l:nD} (3).
On the other hand, by \lemref{l:nD} (1), we have $h^1(\cO_X(-sZ))\ge h^1(\cO_X(-s'Z))$.
Hence $h^1(\cO_X(-sZ)) = p_g(Z^{\bot})$.
This implies that $\br(I)\le (2g-2)/\delta+2$ by \proref{p:brq}.

For the very last inequality, note that $\gon(C) \ge 2$ since $p_a(C)>0$, and $\delta \ge -WC \ge 2$.
\end{proof}

\section{Examples}\label{s:example}

In this section, we present examples and remarks concerning the normal reduction number.
Let $X\to \spec(A)$ be a resolution with exceptional set $E$, as in the preceding section.

\exref{e:br346} below shows that the converse of \proref{p:reh} (2) does not hold 
(see also the last example in \cite{NNO} for another example).
In the proof of the fact, we will use the elliptic sequence defined as follows.

\begin{defn}
Assume that $A$ is an elliptic singularity, namely, $p_f(A)=1$.
Let $C=\mc{(Z_f)}$; this cycle is called the {\em minimally elliptic cycle}.
We define the {\em elliptic sequence} $\{Z_0, \dots, Z_m\}$ 
on $X$ as follows.  
Let $Z_{0}$ be the fundamental cycle on $E$. 
If $Z_0 C<0$, then by definition the elliptic
 sequence is $\{Z_{0}\}$.
If $Z_0, \dots, Z_i$ have been determined and $Z_{i}  C=0$, 
then define $B_{i+1}$ to be the connected component of $Z_i^{\bot}$ containing $\supp(C)$, and let $Z_{i+1}$ be the fundamental cycle on $B_{i+1}$.
If we have $Z_{m} C<0$ for some $m\ge 0$, then the elliptic
 sequence is defined to be $\{Z_{0}, \dots ,Z_{m}\}$.
\end{defn}

\begin{ex}[{cf. \cite[\S 4]{OWY4}}]\label{e:br346}
Let $A=k[[x,y,z]]/(x^3+y^4+z^6)$ and let $X_0\to \spec(A)$ be the minimal resolution with exceptional set $F$.
Then $F$ is a simple normal crossing divisor and  the weighted dual graph of $F=F_0+F_1+F_2+F_3$ is expressed as in \figref{fig:346}, where $F_0$ is an elliptic curve and others are nonsingular rational curves.
\begin{figure}[htb]
\[
\xy
 (0,0)*+{-2}*\cir<10pt>{}="A2"*++!D(-1.5){[1]} *++!UL(1.2){F_0}; 
(-15,0)*+{-2 }*\cir<10pt>{}="A1" *++!L(-2){F_1}; 
(15,0)*+{-2}*\cir<10pt>{}="A3"  *++!L(2){F_3}; 
(0,10)*+{-2}*\cir<10pt>{}="A4" *++!L(-1.5){F_2}; 
\ar @{-} "A2" ;"A1"  
\ar @{-} "A2" ;"A3" 
\ar @{-} "A2" ;"A4"  
 \endxy
\]
\caption{\label{fig:346} The resolution graph of $x^3+y^4+z^6=0$}
\end{figure}
Let $Z_F$ denote the fundamental cycle on $F$; then $Z_F=F+F_0=\mc{(Z_F)}$.
We have $\br(\m)=2$, $p_g(A)=3$, and $p_f(A)=p_a(A)=2$ (see \cite[3.1, 3.10, 4.3]{OWY4}).
We prove that $\br(A)=2$.

Let $I = I_Z$ be an integrally closed $\m$-primary ideal represented by a cycle $Z$ on a resolution $X$. 
Let $f\: X\to X_0$ be the natural morphism, and for a divisor $D$ on $X_0$, $f_*^{-1}D$ denotes the proper transform of $D$.

 We show that $\br (I)\le 2$.
It follows that $\br(I)=1$ if $p_g(Z^{\bot})=p_g(A)=3$, and $\br(I)=2$ if $p_g(Z^{\bot})=2$ by \corref{c:r=2}.
Thus we may assume that $p_g(Z^{\bot})\le 1$. 
Let $W=\mc{(Z_f)}$ and $C\le E$ the elliptic curve. 
Then $W\ge f_*^{-1}Z_F$ by \lemref{l:minpb}.
We also have that
$h^1(\cO_W)=p_a(W) =p_a(Z_f)=2$ by \eqref{eq: h1B}.
Since $p_g(Z^{\bot}) \le 1$, we have $\red W\not\le Z^{\bot}$, that is,
$ZW<0$  (cf. \eqref{eq:trivialW}). 

First, we consider the case $ZC<0$. 
Applying \thmref{t:lambda} to $I=I_Z$, 
it is sufficient to prove that $\lambda(Z,X)=0$.
Let $D\in \cB(Z)^e$ and assume that $-2\chi(D)=2p_a(D)-2>0$.
Then we have $-2\chi(D)=2$ and $D \ge W \ge 2C$.
Since $\cO_C(-Z)$ is also generated, we have $-ZC\ge 2$
(cf. the proof of \thmref{t:ACmain} (1)).
Therefore, $-ZD \ge -Z(2C) \ge 4$ and $\fl{-2\chi(D)/(-ZD)}=0$.
Hence we obtain $\lambda(Z,X)=0$.

Next assume that $ZC=0$. 
Let $B$ be the connected component of $Z^{\bot}$ containing $C$, and let $Z_B$ denote the fundamental cycle on $B$.
Since $p_g(Z^{\bot})=1$, it follows from \eqref{eq:ppp} that $B$ contracts to an elliptic singularity, that is, $\chi(Z_B)=0$.
Let $\{Z_0=Z_B, \dots, Z_m\}$ be the elliptic sequence on $B$,
 and let $D=Z_0+\cdots +Z_m$.
Note that $Z_i$ is reduced at $C$
(cf. \lemref{l:outmc}).
It follows that $D$ is anti-nef on $B$, $\chi(D)=0$, and $Z_mC<0$ (cf. \cite[Theorem (6.4)]{tomari.ell}).
We also have $h^1(\cO_D)=1$ since $h^1(\cO_C)\le h^1(\cO_D)\le p_g(Z^{\bot})$.

Suppose that $H^1(\cO_X(-Z-D))=0$.
Then, from the exact sequence
\[
0 \to \cO_X(-Z-D) \to \cO_X(-Z) \to \cO_D(-Z) \to 0,
\]
we have $q_I(1)=h^1(\cO_X(-Z))=h^1(\cO_D)=1=p_g(Z^{\bot})$.
Hence $\br(I)=2$ by \corref{c:r=2}. 

We will show $H^1(\cO_X(-Z-D))=0$, applying R{\"o}hr's vanishing theorem.

\begin{clm}\label{ck:ZD}
$Z+D$ is anti-nef.
\end{clm}
\begin{proof}[Proof of \clmref{ck:ZD}]
Suppose that $(Z+D)E_i\ge 1$. 
Then $E_i\not\le D$, $DE_i>0$, and $E_i\cong \PP^1$.
We have $ZE_i<0$ by the definition of $B=\red{D}$.
Thus $DE_i\ge 1-ZE_i\ge 2$.
By the definition of $p_a(A)$, for any cycle $D'>0$, $\chi(D')=1-p_a(D') \ge 1-p_a(A)=-1$.
Hence we have 
\[
-1 \le \chi(D+E_i)=\chi(D)+\chi(E_i)-DE_i = 1-DE_i\le -1.
\]
This implies that $DE_i=1-ZE_i=2$, and thus $(Z+D)E_i=1$.
Therefore, $H^j(\cO_{E_i}(-Z-D)) \cong H^j(\cO_{\PP^1}(-1))=0$ for $j=0,1$.
From the exact sequence
\[
0 \to \cO_{E_i}(-Z-D) \to \cO_{D+E_i}(-Z) \to \cO_D(-Z) \to 0,
\]
the condition $\chi(D)=0$, and \eqref{eq:trivialW}, we obtain that
\[
h^0( \cO_{D+E_i}(-Z)) = h^0(\cO_D(-Z)) = h^0(\cO_D)=h^1(\cO_D)=1;
\]
  it contradicts that $\cO_{D+E_i}(-Z)$ is generated and $-Z(D+E_i)=-ZE_i>0$.
Hence $Z+D$ is anti-nef.
\end{proof}

It is enough to prove that for every $V\in \cB$, the following inequality holds:
\begin{equation}
\label{eq:Rineq}
(-Z-D)V > -2\chi(V).
\end{equation}
If $\chi(V)> 0$, then \eqref{eq:Rineq} follows from \clmref{ck:ZD};
if $\chi(V)=0$, then $V\ge C$ and \eqref{eq:Rineq} follows from the inequality $(-Z-D)C = -DC = -Z_m C>0$.
Assume that $\chi(V)=-1$, namely, $p_a(V)=2$.
Then, we have $V\ge W \ge f_*^{-1}Z_F$. 
Therefore, to show \eqref{eq:Rineq}, it is sufficient to prove $(-Z-D)W\ge 3$.
Since $f_*Z_i$ is the fundamental cycle on its support, we may assume that $f_*Z_i$ is one of the cycles $F_0$, $F_0+F_1$, $F_0+F_1+F_2$.
Any $Z_i$ is reduced at every component of $f_*^{-1}F$ since $Z_i\not\ge 2C$.
If $Z_i \ge f_*^{-1}F_j$, then we have 
\begin{equation}
\label{eq:f*F}
Z_if_*^{-1}F_j \le (f^*f_*Z_i)f_*^{-1}F_j = (f_*Z_i)F_j.
\end{equation}

Assume $Z_0C=0$.
By \eqref{eq:f*F}, we have ($f_*Z_0)F_0=0$.
Thus, $f_*Z_0=F_0+F_1+F_2$ and $Z_0f_*^{-1}F_j \le -1$ for $j=1,2$.
Therefore, $f_*Z_1=F_0$, $Z_1 C \le (f_*Z_1)F_0 = F_0^2=-2$,  and $m=1$.
Hence $(-Z-D)W = -DW \ge -Z_1(2C)\ge 4$.
Let us consider the case $m=0$, that is, $Z_0C<0$ and $D=Z_0$.
If $-Z_0C\ge 2$, we have $(-Z-D)W\ge 4$ as above.
Assume that $-Z_0C=1$.
Then $f_*Z_0=F_0+F_1$ or $F_0+F_1+F_2$.
Since $(f^*f_*Z_0-Z_0)(f_*^{-1}f_*Z_0) \ge 0$,  
we have 
$Z_0f_*^{-1}(f_*Z_0) \le (f_*Z_0)^2 =-2$.
Note that $W \ge f_*^{-1}Z_F \ge f_*^{-1}Z_0 +C$.
Therefore, 
\[
(-Z-D)W \ge (-Z-D)(f_*^{-1}(f_*Z_0) + C) 
= -Z_0(f_*^{-1}(f_*Z_0) + C) \ge  2+1 = 3.
\]
Finally, we have proved that $\br(I)=2$.
\end{ex}

The following example shows that $\br(A)$ is not a combinatorial invariant.

\begin{ex}\label{e:non-topol}
Assume that $A$ is a homogeneous hypersurface singularity of degree $d\ge 3$ and $X$ is the minimal resolution. Then $E$ is a nonsingular curve of genus $g=(d-1)(d-2)/2$ with $E^2=-d$, and $\br(A) = d-1$ by \proref{p:reh}.

Next, we consider a special type of singularity, described in \cite[Example 3.9]{OWY5}, whose resolution graph coincides with that of the homogeneous hypersurface singularity mentioned above. 
Let $d\ge 4$ be an even integer and let $g=(d-1)(d-2)/2$.
Suppose $C$ is a hyperelliptic curve of genus $g$, 
 and let $D_0$ be a divisor on $C$ obtained as the pull-back of a point via the double cover $C\to \PP^1$. 
Let $D=(d/2)D_0$, and let  $R=R(C,D)= \bigoplus_{n \ge 0} H^0(X, \cO_C(nD))$. 
Assume that $X\to \spec(R)$ is the minimal resolution with exceptional set $E$.
Then we have $E\cong C$ and $-E^2=\deg D =d$.
Hence the resolution graph of $A$ and $R$ are the same.
However, $\br(R)=g+1>d-1=\br(A)$.  
Note also that $\br(R)$ attains the upper bound given in \thmref{t:ACmain}.
\end{ex}

For rational singularities, elliptic singularities, and almost cone singularities, we have the inequality $\br(A) \le p_f(A)+1$.
However, it does not hold in general.

The following result is an immediate consequence of \cite[Theorem 6.5]{BDMOVWY}.

\begin{prop}\label{p:Zar}
Let $a,b\in \Z$ such that $2 \le a \le b$, and let $g\in (x,y)^b \setminus (x,y)^{b+1} \subset k[[x,y]]$. 
Assume that $A=k[[x,y,z]]/(z^a-g)$.
Then $\br(\m)=\fl{(a-1)b/a}$.
\end{prop}

\begin{ex}
[Tomari {\cite[Example (2.13)]{tomari.max}}]
\label{e:tomari}
Let $A=k[[x,y,z]]/(z^2-(x+y)(x^4+y^6)(x^6+y^4))$.
Then the resolution graph of $A$ is expressed as in \figref{fig:tomari}; $E_1$ and $E_2$ are rational curves, and $E_3$ and $E_4$ are elliptic curves.
\begin{figure}[htb]
\[
\xy
 (0,0)*+{-3}*\cir<10pt>{}="A2"*++!D(-1.5){E_2}; 
(0,10)*+{-2 }*\cir<10pt>{}="A1"*++!R(-1.5){E_1}; 
(-15,0)*+{-1 }*\cir<10pt>{}="A3"*++!D(-1.5){[1]}*+++!L(-1.5){E_3}; 
(15,0)*+{-1 }*\cir<10pt>{}="A4"*++!D(-1.5){[1]}*+++!R(-1.5){E_4}; 
\ar @{-} "A2" ;"A1"  
\ar @{-} "A2" ;"A3" 
\ar @{-} "A2" ;"A4"  
 \endxy
\]
\caption{\label{fig:tomari} The resolution graph of $A$}
\end{figure}
It follows that 
$p_g(A)=8$ and
$p_f(A)=2$. 
On the other hand, by \proref{p:Zar}, we have $\br(\m)=4$.
Hence we obtain 
\[
\br(A) \ge \br(\m) > p_f(A)+1.
\]
\end{ex}

\begin{prob}
Classify
 the singularities $A$ for which $\br(A)\le p_f(A)+1$ 
 and find a suitable lower bound for $\br(A)$.
\end{prob}

%


\end{document}